\documentclass[11pt]{amsart}
\usepackage[babel]{csquotes}
\usepackage{enumitem}
\usepackage{amsmath,amsthm,amssymb,mathrsfs,amsfonts,verbatim,enumitem,color,leftidx,mathtools,bm}
\usepackage{mathabx}
\usepackage{etoolbox} 
\usepackage{bbm}
\usepackage[all,tips]{xy}
\usepackage{graphicx,ifpdf}
\usepackage{stmaryrd}
\usepackage{amssymb}
\usepackage{dsfont}

\ifpdf
\fi
\usepackage[colorlinks]{hyperref}
\hypersetup{
linkcolor=blue,          
citecolor=green,        
}

\usepackage{tikz}
\usetikzlibrary{calc}
\usetikzlibrary{shapes.misc}

\tikzset{cross/.style={cross out, draw=black, minimum size=2*(#1-\pgflinewidth), inner sep=0pt, outer sep=0pt},
cross/.default={1pt}}

\newtheorem{thm}{Theorem}[section]

\newtheorem{cor}[thm]{Corollary}
\newtheorem{prop}[thm]{Proposition}

\newtheorem{conj}{Conjecture}
\newtheorem*{ex*}{Example}

\theoremstyle{definition}

\theoremstyle{remark}
\newtheorem{rem}[thm]{Remark}

\numberwithin{equation}{section}
\definecolor{esperance}{rgb}{0.0,0.5,0.0}

\newcommand{\R}{\mathbb{R}}

\newcommand{\Z}{\mathbb{Z}}

\DeclareMathOperator{\diam}{diam}
\DeclareMathOperator{\card}{card}

\newcommand{\Bad}{\mb{Bad}}
\newcommand{\Sing}{\mb{Sing}}
\newcommand{\DI}{\mb{DI}}
\newcommand{\vol}{\mathrm{vol}}

\newcommand{\al}{\alpha}
\newcommand{\ga}{\gamma}

\newcommand{\del}{\delta}
\newcommand{\Del}{\Delta}
\newcommand{\lam}{\lambda}
\newcommand{\Lam}{\Lambda}
\newcommand{\eps}{\epsilon}

\newcommand{\cC}{\mathcal{C}}

\newcommand{\cH}{\mathcal{H}}

\newcommand{\cX}{\mathcal{X}}

\newcommand{\bR}{\mathbb{R}}
\newcommand{\bZ}{\mathbb{Z}}
\newcommand{\bQ}{\mathbb{Q}}

\newcommand{\bN}{\mathbb{N}}

\newcommand{\SL}{\operatorname{SL}}

\newcommand{\ASL}{\operatorname{ASL}}

\newcommand\wt[1]{\widetilde{#1}}
\newcommand\wh[1]{\widehat{#1}}

\newcommand\idist[1]{\langle#1\rangle}

\newcommand\mb[1]{\mathbf{#1}}

\newcommand\tb[1]{\textbf{#1}}

\newcommand{\onto}{\xymatrix{\ar@{>>}[r]&}}

\newcommand{\eq}[1]
{
\begin{equation*}
{#1}
\end{equation*}
}
\newcommand{\eqlabel}[2]
{
\begin{equation}
{#2}\label{#1}
\end{equation}
}

\makeatletter
\newcommand*{\rom}[1]{\expandafter\@slowromancap\romannumeral #1@}
\makeatother

\begin{document}

\title{Infinitely badly approximable affine forms}

\date{}

\author{Taehyeong Kim}
\address{Department of Mathematics, Brandeis University, Waltham MA}
\email{taehyeongkim@brandeis.edu}

\thanks{}



\def\thefootnote{}
\footnote{The author was supported by the ERC grant HomDyn, ID 833423.}

\begin{abstract}
A pair $(A,\mb{b})$ of a real $m\times n$ matrix $A$ and $\mb{b}\in\bR^m$ is said to be \textit{infinitely badly approximable} if
\[
\liminf_{\mb{q}\in\bZ^n, \|\mb{q}\|\to\infty} \|\mb{q}\|^{\frac{n}{m}}\|A\mb{q}-\mb{b}\|_{\bZ} =\infty,
\]
where $\|\cdot\|_\bZ$ denotes the distance from the nearest integer vector. In this article, we introduce a novel concept of singularity for $(A,\mb{b})$ and characterize the infinitely badly approximable property by this singular property. 
As an application, we compute the Hausdorff dimension of the infinitely badly approximable set. We also discuss dynamical interpretations on the space of grids in $\bR^{m+n}$.
\end{abstract}
\maketitle
\section{Introduction}
\subsection{Background and motivation}\label{subsec_1.1}
Let $m,n$ be positive integers, and denote by $M_{m,n}(\bR)$ the set of $m\times n$ real matrices.
We start by introducing the following classical definition. For any $\eps>0$, we say that $A\in M_{m,n}(\bR)$ is \textit{$\eps$-badly approximable} if 
\[
\liminf_{\mb{q}\in\bZ^n, \|\mb{q}\|\to\infty} \|\mb{q}\|^{\frac{n}{m}}\|A\mb{q}\|_\bZ \geq \eps.
\]
Here and hereafter, $\|\mb{x}\|=\max_{1\leq i\leq k} |x_i|$ and $\|\mb{x}\|_\bZ = \min_{\mb{n}\in\bZ^k}\|\mb{x}-\mb{n}\|$ for $\mb{x}\in\bR^k$.
We say that $A\in M_{m,n}(\bR)$ is \textit{badly approximable} if it is $\eps$-badly approximable for some $\eps>0$.
It is well-known that the set of badly approximable matrices is of zero Lebesgue measure \cite{Khi26}, but has full Hausdorff dimension \cite{Sch69}. See \cite{BK15,Sim18} for the set of $\eps$-badly approximable matrices.

Dirichlet's theorem (1842) says that for any $A\in M_{m,n}(\bR)$ there exist infinitely many $\mb{q}\in\bZ^n$ such that
$\|A\mb{q}\|_\bZ \leq \|\mb{q}\|^{-\frac{n}{m}},$ 
hence there is no $\eps$-badly approximable matrix for any $\eps>1$. However, if we consider inhomogeneous Diophantine approximation, the situation becomes completely different, and investigating this situation is the goal of the present paper.

More precisely, given $\eps>0$, we say that a pair $(A,\mb{b})\in M_{m,n}(\bR)\times \bR^m$ is \textit{$\eps$-badly approximable} if 
\[
\liminf_{\mb{q}\in\bZ^n, \|\mb{q}\|\to\infty} \|\mb{q}\|^{\frac{n}{m}}\|A\mb{q}-\mb{b}\|_{\bZ} \geq \eps.
\]  
Denote by $\Bad(\eps)$ the set of $\eps$-badly approximable pairs in $M_{m,n}(\bR)\times \bR^m$. For $A\in M_{m,n}(\bR)$ and $\mb{b}\in \bR^m$, consider the slices of $\Bad(\eps)$ as follows:
\[\begin{split}
\Bad_A(\eps)&=\{\mb{b}\in\bR^m: (A,\mb{b})\in\Bad(\eps)\},\\ 
\Bad^{\mb{b}}(\eps)&=\{A\in M_{m,n}(\bR): (A,\mb{b})\in\Bad(\eps)\}.
\end{split}\]
A pair $(A,\mb{b})\in M_{m,n}(\bR)\times \bR^m$ is said to be \textit{infinitely badly approximable} if the pair $(A,\mb{b})$ is $\eps$-badly approximable for all $\eps>0$, or equivalently,
\[
\liminf_{\mb{q}\in\bZ^n, \|\mb{q}\|\to\infty} \|\mb{q}\|^{\frac{n}{m}}\|A\mb{q}-\mb{b}\|_{\bZ} = \infty.
\]  
Denote
\[\Bad(\infty)=\bigcap_{\eps>0}\Bad(\eps),\ \Bad_A(\infty)=\bigcap_{\eps>0}\Bad_A(\eps),\ \Bad^{\mb{b}}(\infty)=\bigcap_{\eps>0}\Bad^{\mb{b}}(\eps).
\]

We first remark that if $\mb{b}\in\bZ^m$, we return to the homogeneous Diophantine approximation, so we will usually only consider the case $\mb{b}\in\bR^m\setminus\bZ^m$.
Let us discuss the following motivating observation: For any $A\in M_{m,n}(\bR)$, Kronecker's theorem (see e.g. \cite[Chapter \rom{3}, Theorem \rom{4}]{Cas57}) asserts that the sequence $\left\{A\mb{q} \right\}_{\mb{q}\in\bZ^n}$ modulo $1$ is dense in $\bR^m/\bZ^m$ if and only if the subgroup ${^{t}A}\bZ^m + \bZ^n$ of $\bR^n$ has maximal rank $m+n$ over $\bZ$, where $^t A$ denotes the transpose of the matrix $A$. It follows that if $^t A$ is rational in the sense that ${^{t}A}\mb{y}\in\bZ^n$ for some nonzero $\mb{y}\in\bZ^m$, then there is an open subset $U$ of $\bR^m/\bZ^m$ not including the sequence $\left\{A\mb{q} \right\}_{\mb{q}\in\bZ^n}$ modulo $1$, hence for any $\mb{b}\in U$ the pair $(A,\mb{b})$ is infinitely badly approximable.

Through this observation, one may expect 
\eqlabel{Eq_expect}{\begin{split}
\text{a certain connection between ``singularity" of } ^tA \\ \text{ and infinitely badly approximability of } (A,\mb{b}).\end{split}
}
Let us recall the classical definition of singularity. We say that $A\in M_{m,n}(\bR)$ is \textit{singular} if for any $\eps>0$ for all sufficiently large $X$ the inequalities
\eqlabel{Eq_OldSing}{
\|A\mb{q}\|_\bZ < \eps X^{-\frac{n}{m}} \quad\text{and}\quad 0<\|\mb{q}\|<X
}
have an integer solution $\mb{q}\in\bZ^n$.

One possible way to address \eqref{Eq_expect} comes from \cite{BL05}, although it is not explicitly stated in that paper.
We denote by $w(A,\mb{b})$ the supremum of the real numbers $w$ for which, for arbitrarily large $X$, the inequalities
$$\|A\mb{q}-\mb{b}\|_\bZ < X^{-w} \quad\text{and}\quad\|\mb{q}\| < X$$
have an integer solution $\mb{q}\in\bZ^n$. We also denote by $\wh{w}(A)$ the supremum of the real numbers $w$ for which, for all sufficiently large $X$, the inequalities
$$\|A\mb{q}\|_\bZ < X^{-w} \quad\text{and}\quad 0<\|\mb{q}\| < X$$
have an integer solution $\mb{q}\in\bZ^n$. 
It follows from \cite[Theorem]{BL05} that for almost all $\mb{b}\in\bR^m$,
\[
w(A,\mb{b}) = \frac{1}{\wh{w}({^tA})}.
\] 
If $\wh{w}(^t A) > \frac{m}{n}$, the set of $\mb{b}\in\bR^m$ such that $w(A,\mb{b})<\frac{n}{m}$ is of full Lebesgue measure, hence $\Bad_A(\infty)$ is of full Lebesgue measure. Note that $\wh{w}(^t A) > \frac{m}{n}$ implies that $^tA$ is singular, or equivalently, $A$ is singular. 

The author \cite{Kim} strengthened this observation under a certain weaker assumption. 
We remark that even if $A$ is singular, the set $\Bad_A(\infty)$ may have zero Lebesgue measure (see \cite[Theorem 8.4]{MRS}).

In terms of Hausdorff dimension, it was shown in \cite[Theorem 1.5]{ET11} that if $A\in M_{m,n}(\bR)$ is singular, then the set $\Bad_A(\infty)$ has full Hausdorff dimension (in fact, it is winning on some fractals). 

To the authors' knowledge, these are the only known results in this direction.
In this paper, we will resolve \eqref{Eq_expect}, compute Hausdorff dimension for infinitely badly approximable affine forms, and discuss dynamical interpretations on the space of grids in $\bR^{m+n}$.

\subsection{Main results}\label{subsec_1.2}
One of the main results of this paper is to completely address \eqref{Eq_expect} with the following new definition of singularity:
For fixed $\mb{b}\in\bR^m$, we say that a matrix $^{t}A\in M_{n,m}(\bR)$ is \textit{singular for $\mb{b}$} if 
for any $\eps>0$ for all sufficiently large $X$ the inequalities
\eqlabel{Eq_NewSing}{
\|{^{t}A}\mb{y}\|_\bZ < \eps |\mb{b}\cdot \mb{y}|_\bZ X^{-\frac{m}{n}} \quad \text{and}\quad \|\mb{y}\|< |\mb{b}\cdot\mb{y}|_\bZ X
}
have an integer solution $\mb{y}\in \bZ^m$, where $\cdot$ stands for the usual dot product and $|\cdot|_\bZ$ denotes the distance from the nearest integer. Denote by $\Sing(\mb{b})$ the set of $A\in M_{m,n}(\bR)$ such that $^{t}A$ is singular for $\mb{b}$.

The following theorem can be seen as the transference principle between infinitely badly approximable affine forms and singular linear forms for translation vectors, which addresses \eqref{Eq_expect}.
\begin{thm}\label{Thm_1} 
For any $(A,\mb{b})\in M_{m,n}(\bR)\times \bR^m$, the pair $(A,\mb{b})$ is infinitely badly approximable if and only if  $^tA$ is singular for $\mb{b}$.
In particular, for any $\mb{b}\in\bR^m$,
\[\Bad^{\mb{b}}(\infty) = \Sing(\mb{b})\quad\text{and}\quad \Bad(\infty) = \bigcup_{\mb{b}\in\bR^m} \Sing(\mb{b})\times\{\mb{b}\}.\]  
\end{thm}

\begin{rem}\label{Rem_1}
Since $|\mb{b}\cdot \mb{y}|_\bZ \leq 1/2$ for any $\mb{b}\in\bR^m$ and $\mb{y}\in\bZ^m$, if $^t A$ is singular for $\mb{b}$, then $^t A$ is singular, hence $A$ is singular. 
Therefore, for any $\mb{b}\in\bR^m$, the set $\Bad^{\mb{b}}(\infty)$ is contained in the set of singular matrices.
\end{rem}

It is well-known that the singular set is just rationals in the one dimensional case, but this set is nontrivial in higher dimensional cases. Hence we consider $(m,n)=(1,1)$ and $(m,n)\neq(1,1)$, separately. 
The one dimensional case is simple.
\begin{cor}
When $(m,n)=(1,1)$, for any $b\in\bR\setminus\bZ$, 
\[
\begin{dcases}
    \Bad^b(\infty)= \bQ  &  \text{if } b\notin \bQ,\\
    \Bad^b(\infty)=  \{r/s\in\bQ : \gcd(r,s)=1, q \nmid s\}&  \text{if } b=p/q \text{ with } \gcd(p,q)=1.  
  \end{dcases} \]
\end{cor}
\begin{proof}
Using Theorem \ref{Thm_1}, it is enough to show that $\Sing(b)$ is equal to the right hand side, and it can be easily checked. 
\end{proof}

Now assume $(m,n)\neq (1,1)$. If $\mb{b}\in\bR^m\setminus\bZ^m$, then there exists nonzero $\mb{y}\in\bZ^m$ such that $|\mb{b}\cdot\mb{y}|_\bZ > 0$. Since the set of $A$ such that $^tA \mb{y} \in\bZ^n$ is contained in $\Sing(\mb{b})$, we have
\[
\dim_H \Sing(\mb{b}) \geq \dim_H\{A\in M_{m,n}(\bR): {^tA\mb{y}}\in\bZ^n\} = n(m-1).
\]
Here and hereafter, $\dim_H$ refers to the Hausdorff dimension.
Theorem \ref{Thm_1} implies that 
\[
\dim_H \Bad^{\mb{b}}(\infty) \geq n(m-1) \quad\text{for any }\mb{b}\in\bR^m\setminus\bZ^m,
\]
and hence
\[
\dim_H \Bad(\infty) \geq n(m-1)+m.
\]
On the other hand, using Hausdorff dimension of singular sets \cite{C11, CC16, KKLM17, DFSU}, it follows from Remark \ref{Rem_1} that
\[
\dim_H \Bad^{\mb{b}}(\infty) \leq mn-\frac{mn}{m+n} \ \text{ and }\  \dim_H \Bad(\infty) \leq mn+m-\frac{mn}{m+n}. 
\]
It is natural to ask the exact Hausdorff dimension for $\Bad^\mb{b}(\infty)$ and $\Bad(\infty)$ and our main application of Theorem \ref{Thm_1} answers this question in the case that $m=1$ and $n\geq 2$.
\begin{thm}\label{Thm_2} Assume that $m=1$ and $n\geq 2$.
For any $b\in\bR\setminus\bZ$,
\[\dim_H \Bad^{b}(\infty) = \frac{n^2}{n+1}  \quad \text{and}\quad \dim_H \Bad(\infty) = \frac{n^2}{n+1}+1.\]
Moreover, if $b\in\bQ\setminus\bZ$, there is a constant $C>0$ such that for any $t>1$ and all small enough $\eps>0$,
\[
\frac{n^2}{n+1}+ \eps^{t}\leq \dim_H \Bad^{b}(1/\eps) \leq \frac{n^2}{n+1}+C\eps^{\frac{1}{2}}.
\]
\end{thm}
It is worth noting that Theorem \ref{Thm_2} is somewhat similar to \cite[Theorems 1.1 and 1.3]{CC16}. Indeed, our method relies on their method. This will be discussed at the end of the introduction.

In the view of Theorem \ref{Thm_2}, we may expect 
\begin{conj}\label{Conj_1} 
When $(m,n)\neq(1,1)$, for any $\mb{b}\in\bR^m\setminus\bZ^m$,
\[
\dim_H \Bad^{\mb{b}}(\infty) = mn-\frac{mn}{m+n}  \ \text{ and }\   \dim_H \Bad(\infty) = mn+m-\frac{mn}{m+n}. 
\]
\end{conj}

Now we focus on the set $\Bad_A(\infty)$. Following \cite{DFSU}, we say that $A\in M_{m,n}(\bR)$ is \textit{very singular} if $\wh{w}(A)>\frac{n}{m}$. It can be easily checked that $A$ is very singular if and only if $^tA$ is very singular (see Proposition \ref{Prop_transf}). Following the discussion in Subsection \ref{subsec_1.1}, \cite[Theorem]{BL05} implies that if $A$ is very singular, then $\Bad_A(\infty)$ is of full Lebesgue measure. The following theorem says that the complement of $\Bad_A(\infty)$ should be small with respect to $\wh{w}({^tA})$ in terms of Hausdorff dimension.
\begin{thm}\label{Thm_3}
For any $A\in M_{m,n}(\bR)$,
\[
\dim_H (\bR^m\setminus \Bad_A(\infty)) \leq m-\frac{\wh{w}({^t}A)-\frac{m}{n}}{\wh{w}({^t}A)+1}.
\]
In particular, if $A$ is very singular, or equivalently, $^tA$ is very singular, then the complement of $\Bad_A(\infty)$ cannot have full Hausdorff dimension.
\end{thm}
\begin{rem}\
\begin{enumerate}
\item In the view of \cite{ET11, Kim} and Theorem \ref{Thm_3}, the more singular $A$ is, the larger $\Bad_A(\infty)$ becomes.
\item It seems interesting to give a lower bound for the Hausdorff dimension of the complement of $\Bad_A(\infty)$ or calculate its dimension precisely.
\item It was proved in \cite{BKLR, KKL} that $A$ is \textit{singular on average} if and only if there exists $\eps>0$ such that 
$\Bad_A(\eps)$ has full Hausdorff dimension. 
Thus it seems interesting to reveal a certain equivalence between a Diophantine property of $A$ and a metrical property of $\Bad_A(\infty)$.
\end{enumerate}
\end{rem}

\subsection{Dynamical discussion}\label{subsec_1.3}
In this subsection, we will discuss the dynamical interpretation of infinitely badly approximability following classical relations between homogeneous dynamics and Diophantine approximation.
 
Consider the homogeneous space $\cX_0=\SL_{m+n}(\bR)/\SL_{m+n}(\bZ)$, which can be identified with the space of unimodular lattices in $\bR^{m+n}$. 
For any $t\in\bR$ and $A\in M_{m,n}(\bR)$, let us denote
\[
a_t = \begin{pmatrix} e^{t/m}I_m & \\ & e^{-t/n}I_n \end{pmatrix}\quad\text{and}\quad 
u_A= \begin{pmatrix} I_m & A \\ & I_n \end{pmatrix}.
\]
Define the function $\Del_0:\cX_0\to\bR_{>0}$ by $\Del_0(x)=\min_{v\in x\setminus\{0\}}\|v\|$. 
Dani's correspondence \cite[Theorem 2.14]{D85} says that 
\[
A \text{ is singular} \iff \Del_0(a_t u_A\bZ^{m+n}) \to 0 \text{ as } t\to\infty.
\]
Let $\ASL_{m+n}(\bR)=\SL_{m+n}(\bR)\ltimes\bR^{m+n}$ and $\ASL_{m+n}(\bZ)=\SL_{m+n}(\bZ)\ltimes\bZ^{m+n}$. The homogeneous space $\cX=\ASL_{m+n}(\bR)/\ASL_{m+n}(\bZ)$ can be identified with the space of unimodular grids in $\bR^{m+n}$, i.e. affine shifts of unimodular lattices in $\bR^{m+n}$. Define the function $\Del : \cX \to \bR_{\geq 0}$ by $\Del(x) = \min_{v\in x}\|v\|$.
For $A\in M_{m,n}(\bR)$ and $\mb{b}\in \bR^m$, consider the grid 
\[
\Lam_{A,\mb{b}}= u_A \bZ^{m+n}+ \begin{psmallmatrix} \mb{b} \\ \mb{0} \end{psmallmatrix} \in \cX.
\]
The diagonal matrix $a_t$, viewed as a linear transformation on $\bR^{m+n}$, induces a natural action on $\cX$.

Following \cite[Subsection 1.3]{Kle99}, we say that $(A,\mb{b})\in M_{m,n}(\bR)\times\bR^m$ is \textit{rational} if $\|A\mb{q}-\mb{b}\|_\bZ = 0$ for some $\mb{q}\in\bZ^n$, and \textit{irrational} otherwise.  
Using the proof of \cite[Theorem 4.4]{Kle99}, we have the following characterization:
\eq{
(A,\mb{b})\in \Bad(\infty)\text{ and irrational} \iff \Del(a_t \Lam_{A,\mb{b}}) \to \infty \text{ as } t\to\infty.
}
As stated in \cite[Subsection 1.3]{Kle99}, if $(A,\mb{b})$ is rational, then the badly approximability of $(A,\mb{b})$ is the same as that of $A$. Since there is no infinitely badly approximable linear form, we indeed have 
\eqlabel{Eq_char}{
(A,\mb{b})\in \Bad(\infty) \iff \Del(a_t \Lam_{A,\mb{b}}) \to \infty \text{ as } t\to\infty.
}
Any lattice can be seen as a grid containing the origin, so we may consider $\cX_0$ as a subset of $\cX$. In particular, $\Del^{-1}(0) = \cX_0$. Therefore, the dynamical property that $\Del(a_t x) \to\infty$ as $t\to\infty$ means that the orbit $a_t x$ moves away from $\cX_0$ as $t\to\infty$.

The dynamical interpretation of Remark \ref{Rem_1}, which is weaker than Theorem \ref{Thm_1}, is the following (see Proposition \ref{Prop_transf}):
\eq{
\Del(a_t x)\to\infty \text{ as } t\to\infty \implies \Del_0(a_t\pi(x))\to 0 \text{ as } t\to\infty,
}
where $\pi:\cX \to \cX_0$ denotes the natural projection sending a grid $x = x_0 + v$ to its underlying lattice $x_0$; for this projection, see e.g. \cite[Subsection 2.4]{MRS}.
Using this implication, it follows from \cite[Theorem 1.1]{KKLM17} that 
\eqlabel{Eq_upper}{ \dim_H \{x \in \cX : \Del(a_tx)\to \infty \text{ as }t\to\infty \} \leq \dim_H \cX - \frac{mn}{m+n}.}
As an application of Theorem \ref{Thm_2} and this upper bound, we have
\begin{thm}\label{Thm_4}
When $m=1$ and $n\geq 2$, 
\[ \dim_H \{x \in \cX : \Del(a_tx)\to \infty \text{ as }t\to\infty \} = \dim_H \cX - \frac{n}{n+1}.\]
\end{thm}
Similar to Conjecture \ref{Conj_1}, we may expect
\begin{conj}\label{Conj_2}
When $(m,n)\neq (1,1)$,
\[ \dim_H \{x \in \cX : \Del(a_tx)\to \infty \text{ as }t\to\infty \} = \dim_H \cX - \frac{mn}{m+n}.\]
\end{conj}

\subsection{Discussion of the proof of Theorem \ref{Thm_2}}
In this paper, although Theorem \ref{Thm_1} is philosophically significant, the most technically important part is Theorem \ref{Thm_2}. Thus, we will discuss the proof of Theorem \ref{Thm_2}.

Thanks to Theorem \ref{Thm_1}, it is enough to estimate the Hausdorff dimension of $\Sing(\mb{b})$. 
For this, we will modify the method of \cite{CC16} as follows.
Comparing the new singular definition \eqref{Eq_NewSing} with the classical one \eqref{Eq_OldSing}, we observe a difference of $|\mb{b}\cdot\mb{y}|_\bZ$-term. Thus, if we consider the set of $A$ in $M_{m,n}(\bR)$ such that for any $\eps>0$ for all sufficiently large $X$ the inequalities
\eqlabel{Eq_Idea}{
\|^tA\mb{y}\|_\bZ<\eps \del X^{-\frac{m}{n}},\quad \|\mb{y}\|<\del X,\quad\text{and} \quad
|\mb{b}\cdot\mb{y}|_\bZ>\del
} have an integer solution $\mb{y}\in\bZ^m$, then this set is contained in $\Sing(\mb{b})$. We can also consider this set as the set of singular matrices with solutions restricted on the set $\{\mb{y}\in\bZ^m : |\mb{b}\cdot\mb{y}|_\bZ >\del\}$. Therefore, we will modify the fractal structure developed in \cite{CC16} so that the modified structure is contained in the singular set with restricted solutions, and show that this structure has the same dimension as the original structure. 

Now let us discuss why we need the condition $b\in \bQ\setminus\bZ$ in the second argument of Theorem \ref{Thm_2}. In order to bound the Hausdorff dimension of $\Bad^b(1/\eps)$, we need to bound the Hausdorff dimension of the set of $\eps$-Dirichlet improvable vectors for $b$, which will be defined in Section \ref{sec2}. For this, $\del$ in \eqref{Eq_Idea} should be fixed and it is the main difference from the case $\Sing(b)$. If $b=c/d$ with $\gcd(c,d)=1$ and $d\geq 2$, then $|bk|_\bZ<1/d$ if and only if $|bk|_\bZ=0$ for any $k\in\bZ$. This property eventually enable us to bound the Hausdorff dimension of $\Bad^b(1/\eps)$ even if $\del$ in \eqref{Eq_Idea} is fixed. 

A similar idea to the above was used significantly in \cite{KK22}. Of course, the dimension of the singular matrices was calculated in \cite{DFSU}, but the fractal structure remains unclear since they used the variational principle in the parametric geometry of numbers. On the other hand, the fractal structure was clearly given in \cite{CC16} so that we can modify it.  
If one can reveal the fractal structure in \cite{DFSU}, it seems plausible to modify the structure as above to obtain Conjectures \ref{Conj_1} and \ref{Conj_2}.

\vspace{3mm}
\textit{Structure of the paper.} 
In Section \ref{sec2}, we recall the transference principle from \cite{Cas57} and establish Theorem \ref{Thm_1}.
Section \ref{sec3:pre} provides the necessary preliminaries for Section \ref{sec3}.
In Section \ref{sec3}, we refine the fractal structures of \cite{CC16} in order to prove Theorem \ref{Thm_2}.
Section \ref{sec4} is devoted to the proof of Theorem \ref{Thm_3}.
Finally, Section \ref{sec5} discusses the dynamical aspects, including the proof of Theorem \ref{Thm_4}.

\vspace{3mm}
\textit{Convention.} In what follows, the notation $A \ll B$ means that there exists a constant $C$ (the \textit{implied constant}) such that $A \leq CB$. The notation $A \asymp B$ means $A \ll B\ll A$.

\vspace{3mm}
\tb{Acknowledgments}. 
I thank Elon Lindenstrauss for his valuable comments and for suggesting Theorem \ref{Thm_3}, Yitwah Cheung for generously sharing the idea of \cite{CC16}, and the anonymous referees for their careful reading of this article and helpful comments.

\section{Transference Principle}\label{sec2}
In this section, we will prove a certain stronger transference argument than Theorem \ref{Thm_1}.
For this, we introduce the following definition:
For given $\eps>0$ and $\mb{b}\in\bR^m$, we say that $^tA \in M_{n,m}(\bR)$ is \textit{$\eps$-Dirichlet improvable for $\mb{b}$} if for all sufficiently large $X$ the inequalities
\[
\|{^{t}A}\mb{y}\|_\bZ < \eps |\mb{b}\cdot \mb{y}|_\bZ X^{-\frac{m}{n}} \quad \text{and}\quad \|\mb{y}\|<|\mb{b}\cdot\mb{y}|_\bZ X
\]
have an integer solution $\mb{y}\in \bZ^m$. Denote by $\DI_\eps(\mb{b})$ the set of $A\in M_{m,n}(\bR)$ such that $^tA$ is $\eps$-Dirichlet improvable for $\mb{b}$. Note that $\Sing(\mb{b})=\displaystyle{\bigcap_{\eps>0}}\DI_\eps(\mb{b})$.
The following is the main theorem of this section.
\begin{thm}\label{Thm_Sec2}
There are constants $c_1, c_2 >0$ such that for any $\eps>0$ and any $(A,\mb{b})\in M_{m,n}(\bR)\times\bR^m$ the following holds:
If $^tA$ is $c_1\eps^{m/n}$-Dirichlet improvable for $\mb{b}$, then  $(A,\mb{b})$ is $1/\eps$-badly approximable. On the other hand, if $(A,\mb{b})$ is $1/\eps$-badly approximable, then $^tA$ is $c_2\eps^{m/n}$-Dirichlet improvable for $\mb{b}$. In particular, for any $\mb{b}\in\bR^m$
\[
\DI_{c_1 \eps^{\frac{m}{n}}}(\mb{b}) \subset \Bad^{\mb{b}}(1/\eps)\subset \DI_{c_2 \eps^{\frac{m}{n}}}(\mb{b}).
\]
\end{thm}
\begin{proof}[Proof of Theorem \ref{Thm_1}]
Taking intersection over all $\eps>0$, Theorem \ref{Thm_Sec2} implies Theorem \ref{Thm_1}.
\end{proof}
The following transference principle is the main ingredient of the proof of Theorem \ref{Thm_Sec2}.
\begin{thm}\label{Thm_TP}\cite[Theorem \rom{17} in Chapter \rom{5}]{Cas57}
Let $A\in M_{m,n}(\bR)$, $\mb{b}\in\bR^m$, $C>0$, $X>1$ be given.
\begin{enumerate}
\item\label{Item_Nec} A necessary condition that 
\eqlabel{Eq_Inhomo}{
\|A\mb{q}-\mb{b}\|_\bZ \leq C\quad\text{and}\quad \|\mb{q}\|\leq X,
} for some $\mb{q}\in\bZ^n$ is that
\eqlabel{Eq_homo}{
|\mb{b}\cdot \mb{y}|_\bZ \leq \gamma \max\left( X \|{^t A}\mb{y}\|_\bZ, C\|\mb{y}\|\right),
} hold for all $\mb{y}\in\bZ^m$ with $\gamma=m+n$.
\item\label{Item_Suf} A sufficient condition that \eqref{Eq_Inhomo} has a solution $\mb{q}\in\bZ^n$ is that \eqref{Eq_homo} hold for all $\mb{y}\in\bZ^m$ with $\gamma=2^{m-1}((m+n)!)^{-2}$.
\end{enumerate}
\end{thm}

\begin{proof}[Proof of Theorem \ref{Thm_Sec2}]
Observe that $(A,\mb{b})$ is $1/\eps$-badly approximable if and only if for all large enough $T>1$, there is no solution $\mb{q}\in\bZ^n$ such that
\[
\|A\mb{q}-\mb{b}\|_\bZ \leq \eps^{-1} T^{-\frac{n}{m}}\quad\text{and}\quad \|\mb{q}\|\leq T.
\]
By Theorem \ref{Thm_TP} \eqref{Item_Suf}, this implies that
for all large enough $T>1$, there exists $\mb{y}\in\bZ^m$ such that
\[
\|{^t A}\mb{y}\|_\bZ < |\mb{b}\cdot \mb{y}|_\bZ \ga^{-1}T^{-1}\quad\text{and}\quad \|\mb{y}\|<|\mb{b}\cdot \mb{y}|_\bZ \ga^{-1}\eps T^{\frac{n}{m}}.
\]
If we substitute $\ga^{-1}\eps T^{\frac{n}{m}}=X$, then we have $\ga^{-1}T^{-1} = \ga^{-1-\frac{m}{n}}\eps^{\frac{m}{n}}X^{-\frac{m}{n}}$, hence, $^{t}A$ is $c_2 \eps^{\frac{m}{n}}$-Dirichlet improvable for some constant $c_2>0$. 

Suppose that $A\in \DI_{c_1\eps^{\frac{m}{n}}}(\mb{b})$ with some constant $c_1>0$ to be determined, that is, for all large enough $X>1$, there exists $\mb{y}\in\bZ^m$ such that
\[
\|{^t A}\mb{y}\|_\bZ < c_1\eps^{\frac{m}{n}} |\mb{b}\cdot \mb{y}|_\bZ  X^{-\frac{m}{n}}\quad\text{and}\quad \|\mb{y}\|<|\mb{b}\cdot \mb{y}|_\bZ X,
\]
or equivalently,
\[
|\mb{b}\cdot\mb{y}|_\bZ > \ga \max\left(\ga^{-1}c_1^{-1}\eps^{-\frac{m}{n}}X^{\frac{m}{n}} \|{^t A}\mb{y}\|_\bZ, \ga^{-1}X^{-1} \|\mb{y}\|\right).
\]
By Theorem \ref{Thm_TP} \eqref{Item_Nec}, there is no solution $\mb{q}\in\bZ^n$ such that
\[
\|A\mb{q}-\mb{b}\|_\bZ \leq \ga^{-1}X^{-1}\quad\text{and}\quad \|\mb{q}\|\leq \ga^{-1}c_1^{-1}\eps^{-\frac{m}{n}}X^{\frac{m}{n}}.
\]
By substituting $\ga^{-1}c_1^{-1}\eps^{-\frac{m}{n}}X^{\frac{m}{n}} = T$ and taking $c_1 = \ga^{-1-\frac{m}{n}}$, we have
\[
\|A\mb{q}-\mb{b}\|_\bZ \leq \eps^{-1} T^{-\frac{n}{m}} \quad\text{and}\quad \|\mb{q}\|\leq T.
\]
Therefore, $(A,\mb{b})$ is $1/\eps$-badly approximable.
\end{proof}

\section{Preliminaries for $\dim_H \Bad^{b}(\infty)$}\label{sec3:pre}
Throughout Sections \ref{sec3:pre} and \ref{sec3}, we assume $m=1$ and $n\geq 2$, and establish Theorem \ref{Thm_2}. 
By the transference argument (Theorem \ref{Thm_Sec2}), it suffices to estimate $\dim_H \Sing(b)$ and $\dim_H \DI_\eps(b)$ for fixed $b \in \R \setminus \Z$. Our approach essentially follows \cite{CC16}. However, since $\Sing(b)$ is strictly smaller than the set of singular vectors (see Remark \ref{Rem_1}), we need to construct a refined fractal structure, smaller than that of \cite{CC16}. This construction will be carried out in Section \ref{sec3}. 

This section is devoted to the necessary preliminaries for this purpose.

\subsection{Self-similar coverings}
We first recall the definition of self-similar structures and some related dimension results following \cite{CC16}.

We equip $\bR^n$ with the metric induced by the Euclidean norm $\|\cdot\|_2$. 
A \textit{self-similar structure} on $\bR^n$ is a triple $(J,\sigma,B)$, where $J$ is countable, $\sigma\subseteq J\times J$, and $B$ is a map from $J$ into the set of bounded subsets of $\bR^n$. A \textit{$\sigma$-admissible sequence} is a sequence $(x_k)_{k\in\bN}$ in $J$ such that 
for any $k$, $(x_k,x_{k+1})\in \sigma$.

For a subset $S$ of $\bR^n$, a \textit{self-similar covering} of $S$ is a self-similar structure $(J,\sigma,B)$ such that, for all $\theta$ in $S$, there exists a $\sigma$-admissible sequence $(x_k)_{k\in\mathbb{N}}$ in $J$ such that
\begin{itemize}
    \item $\lim_{k\to \infty}\diam B(x_k)=0$;
    \item $\bigcap_{k\in \mathbb{N}}B(x_k)=\{\theta\}$.
\end{itemize}
In this case, the self-similar structure $(J,\sigma,B)$ is said to cover the subset $S$.
Given $x\in J$, we denote 
\[
\sigma(x)=\{y\in J : (x,y)\in \sigma\}.
\]
A triple $(J,\sigma,B)$ is said to be a \textit{strictly nested self-similar structure} on $\bR^n$ if it is a self-similar structure on $\bR^n$ such that 
\begin{itemize}
    \item for all $x\in J$, $\sigma(x)$ is finite, $B(x)$ is a nonempty compact subset of $\bR^n$, and for all $y\in\sigma(x)$ we have $B(y)\subset B(x)$;
    \item for each $\sigma$-admissible sequence $(x_k)_{k\in\mathbb{N}}$, $\lim_{k\to\infty}\diam B(x_k)=0$;
    \item for each $x\in J$ and each $y\in\sigma(x)$, $\diam B(y)<\diam B(x)$.
\end{itemize}
We need the following dimension result.
\begin{thm}\cite[Theorem 3.6]{CC16}
\label{thm:CCLow3.6}
  Let $S$ be a subset of $\bR^n$.
  Suppose that there is a strictly nested self-similar structure $(J,\sigma,B)$ that covers a subset of $S$, a subset $J_0\subset J$ that contains a tail of any $\sigma$-admissible sequence, a function $\rho:J\to (0,1)$, and two constants $c,s\geq 0$ such that
  \begin{enumerate}
      \item\label{item_wlower_1} for each $x\in J_0$ and each $y\in\sigma(x)$, there are at most $c$ points $z$ in $\sigma(x)\setminus\{y\}$ such that
      \[
    d(B(y),B(z))\leq\rho(x)\diam B(x);
      \]
      \item\label{item_wlower_2} for every $x \in J_0$,
      \[
\sum_{y\in\sigma(x)}\left(\rho(y)\diam B(y)\right)^s \geq (c+1)\left(\rho(x)\diam B(x)\right)^s;
      \]
      \item\label{item_wlower_3} for all $\alpha\in J_0$ and all $y\in\sigma(x)$, $\rho(y)\diam B(y)<\rho(x)\diam B(x)$.
      \end{enumerate}
      Then $S$ contains a subset of positive $s$-dimensional Hausdorff measure.
\end{thm}
We have introduced the notion of self-similar coverings and discussed their main properties. To make these ideas more concrete, we now give a simple example.

\medskip
\noindent\textbf{Example} (The middle third Cantor set). Let $E_0=[0,1]$ and $E_k = \frac{1}{3}E_{k-1} \cup \left(\frac{2}{3}+\frac{1}{3}E_{k-1}\right)$ for $k\geq 1$. Then the middle third Cantor set is given by $F=\bigcap_{k=0}^\infty E_k$. 
Since $E_k$ consists of $2^k$ intervals of length $3^{-k}$, denoted by $E_{k,j}$ for $j=1,\dots,2^k$ in the natural left-to-right order, we define the self-similar structure by
\begin{align*}
J &= \{(k,j): k\geq 0,\ j=1,\dots,2^k\}, \\
\sigma(k,j) &= \{(k+1,2j-1),(k+1,2j)\}, \\
B(k,j) &= E_{k,j}.
\end{align*}
Then the triple $(J,\sigma,B)$ is a strictly nested self-similar covering of $F$. 
Moreover, the conditions \eqref{item_wlower_1}, \eqref{item_wlower_2}, and \eqref{item_wlower_3} 
of Theorem~\ref{thm:CCLow3.6} are satisfied with $c=0$, $\rho<\tfrac{1}{3}$, and 
$s=\tfrac{\log 2}{\log 3}$. 
Therefore, $\dim_H F \geq \tfrac{\log 2}{\log 3}$, and in fact equality holds 
(see, e.g., \cite[Example~3.7]{Fal}).

\subsection{Fractal structure in \cite{CC16}}
In this subsection, we recall notation and the fractal structure given in \cite{CC16}. Note that our ambient space is $\bR^n$ while the ambient space of \cite{CC16} is $\bR^d$. Therefore, $d$ in \cite{CC16} should be changed to $n$ in our case. Since we do not consider the result in \cite[Section 7]{CC16}, we only consider the fractal structure of \cite[Section 6]{CC16} in the case $x=y$.
 
Let 
\[
Q=\{(p_1,\dots,p_n,q)\in \bZ^{n+1}: \gcd(p_1,\dots,p_n,q)=1, q>0\}.
\]
Given $x=(p,q)=(p_1,\dots,p_n,q)\in Q$, we denote
\[
|x|=q\qquad\text{and}\qquad \wh{x}=\frac{p}{q},
\]
and define the \textit{Farey lattice}
\[
\Lam_x = \bZ^n +\bZ\wh{x} = \pi_x (\bZ^{n+1}),
\]
where $\pi_x:\bR^{n+1}\to\bR^n$ is the map given by $\pi_x(r,s)=r-s\wh{x}$ for $(r,s)\in\bR^n\times\bR$.
Note that $\Lam_x$ is a lattice in $\bR^n$ with the covolume $\vol\ \Lam_x = |x|^{-1}$. Here and hereafter, $\vol\ \Lam$ stands for the covolume of the lattice $\Lam$.
For each $i=1,\dots,n$, denote the $i$-th successive minimum of $\Lam_x$ by $\lam_i(x)$ with respect to the Euclidean norm $\|\cdot\|_2$ on $\bR^n$, and the normalized $i$-th successive minimum by
$\wh{\lam}_i(x)=|x|^{1/n}\lam_i(x)$. 

For each $x\in Q$, let $\Lam_x'$ be a codimension $1$ sublattice of $\Lam_x$ with minimal covolume, $H_x'$ be the real span of $\Lam_x'$, and $H_x = \pi_x^{-1}H_x'$. 
Given $x\in Q$ and a primitive $\al\in \Lam_x$, we let
\[
\Lam_{\al^\perp} = \pi_{\al}^\perp (\Lam_x),
\]
where $\pi_\al^\perp$ is the orthogonal projection of $\bR^n$ onto $\al^\perp$, the subspace of vectors of $\bR^n$ orthogonal to $\al$.
For any $y\in Q$, let us denote
\[
|x\wedge y| = |x||y| \|\wh{x} - \wh{y}\|_2.
\]
For $y\in Q$ with $\pi_x (y)=\al$, the covolume of $\Lam_{\al^\perp}$ is given by 
$$\vol\ \Lam_{\al^\perp} = \frac{\vol\ \Lam_x}{\|\al\|_2}= \frac{1}{|x|\|\pi_x(y)\|_2}=\frac{1}{|x\wedge y|}.$$
As before, we denote the first successive minimum of $\Lam_{\al^\perp}$ by $\lam_1(\al)$ and its normalized version by $\wh{\lam}_1(\al) = |y\wedge z|^{1/(n-1)}\lam_1(\al)$.
Given $\eps>0$ and $x\in Q$, let
\[
\Lam_x(\eps) = \{\al\in\Lam_x : \al \text{ is primitive and } \wh{\lam}_1(\al) >\eps\}.
\]

We fix a coset $H'$ of $H_x'$ such that $\Lam_x\cap H' \neq \varnothing$ and $\Lam_x\setminus H_x' \subset \bigcup_{k\in\bZ}kH'$ as in \cite[Subsection 6.1]{CC16}, and denote by $\al_x^{\perp}$ the unique element of $H'$ that is perpendicular to $H_x'$.
Let $H_x'(k)$ be the coset $k\al_x^\perp + H_x'$
and let $\cC(x)$ be the set of vectors in $\bR^n$ whose angle with $\al_x^\perp$ is at most $\tan^{-1}A_n$, where the constant $A_n$ is in the proof of \cite[Lemma 8.6]{CC16}. Let
\[
\cC_k'(x)=\cC(x)\cap H_x'(k) \quad\text{and}\quad \cC_N(x)=\bigcup_{k=1}^N \cC_k'(x).
\]
For $N\in\bN$, $\eps>0$, and $x\in Q$, define
\[
F_N (x,\eps) = \bigcup_{\al \in \Lam_x(\eps)\cap \cC_N(x)}\zeta(x,\al,\eps),
\]
where $\zeta(x,\al,\eps)=\left\{y\in Q: \pi_x(y)=\al, |y| \in \frac{|x\wedge y|^{n/(n-1)}}{\eps^{n/(n-1)}}(1,2)\right\}.$

The following fractal structure was established in \cite{CC16} to estimate the lower bound of the Hausdorff dimension of the $\eps$-Dirichlet improvable set.
\begin{prop}\cite[Proposition 6.6]{CC16}\label{Prop_CCfrac}
Let 
$$\sigma_{\eps,N}(x)=F_N(x,\eps), \quad B(x)=B\left(\wh{x},\frac{\lam_1(x)}{2|x|}\right), \quad Q_{\eps,N}=\bigcup_{x\in Q}\sigma_{\eps,N}(x).$$ The triple $(Q_{\eps,N},\sigma_{\eps,N},B)$ is a strictly nested self-similar structure covering a subset of the $2\eps$-Dirichlet improvable set provided $\eps$ is small enough.
\end{prop}
Cheung and Chevallier were able to bound from below the Hausdorff dimension of this fractal structure by establishing spacing and local finiteness properties of this structure \cite[Subsection 6.3]{CC16} and controlling the distribution of Farey lattices with bounded distortion \cite[Section 8]{CC16}. 

We will need the following slight generalization of \cite[Proposition 6.11]{CC16} concerning the distribution of Farey lattices with bounded distortion. The proof is essentially the same as in \cite{CC16}, but whereas \cite[Proposition 6.11]{CC16} is stated under the assumption $t>n$, our application requires the case $t\leq n$. For the sake of completeness, we provide the proof.
\begin{prop}\cite[Proposition 6.11]{CC16}\label{Prop_CC16_6.11}
    For any real number $t$, if $\eps$ is small enough, then for all $N\geq 1$ and $x\in Q$ we have
    \[
    S_1(N,t) \gg \sum_{n=1}^N \frac{1}{k^{1+(t-n)\frac{n}{n-1}}},
    \]
    where 
    \[
    S_1(N,t)=\sum_{k=1}^{N}\frac{1}{k^{1+(t-n)\frac{n}{n-1}}}\frac{\card(\cC_k'(x)\cap \Lam_x(\eps))}{\card(\cC_k'(x)\cap\Lam_x)}.
    \]
\end{prop}
\begin{proof}
By \cite[Proposition 8.1]{CC16}, there is a constant $C$ depending only on the dimension $n$ such that
for all $k\geq 1$ and $x\in Q$ we have
\[
\frac{\card(\cC_k'(x)\cap \Lam_x(\eps))}{\card(\cC_k'(x)\cap\Lam_x)} \gg \frac{\phi(k)}{k}\left(1-C\eps^{n-1}\frac{D_1(k)}{k}\right),
\]
where $\phi(k)$ is the Euler function and $D_1(k)=\sum_{\ell|k}\ell$.
Write $\del =(t-n)\frac{n}{n-1}$. It follows from the proof of \cite[Proposition 6.11]{CC16} that
\[
S_1(N,t)\gg \sum_{k=1}^N \frac{1}{k^{1+\del}}\frac{\phi(k)}{k}\left(1-C\eps^{n-1}\frac{D_1(k)}{k}\right) \gg (1-C_t \eps^{n-1})\sum_{k=1}^N\frac{1}{k^{1+\del}}
\]
for some absolute constant $C_t$ depending on $t$. Hence, if $\eps$ is small enough (depending on $t$), the proposition follows.
\end{proof}

\section{Dimension estimates for $\Bad^{b}(\infty)$}\label{sec3}
The main goal of this section is to prove the following theorem, which immediately yields Theorem~\ref{Thm_2}.
\begin{thm}\label{Thm_Sec3} 
For any $b\in\bR\setminus\bZ$,
\eqlabel{Eq_dim_sing}{
\dim_H \Sing(b) \geq\frac{n^2}{n+1}.
}
Moreover, if $b\in\bQ\setminus\bZ$, then for any $t>n$ and all small enough $\eps>0$, 
\eqlabel{Eq_dim_DI}{
\dim_H \DI_{\eps}(b) \geq  \frac{n^2}{n+1} +\eps^t.
} 
\end{thm}

\begin{proof}[Proof of Theorem \ref{Thm_2}]
Combining Theorem \ref{Thm_Sec2}, Theorem \ref{Thm_Sec3}, and \cite[Theorems 1.1 \& 1.3]{CC16}, Theorem \ref{Thm_2} follows.
\end{proof}

\subsection{Fractal structure for $\Sing(b)$}
In this subsection, we fix $b\in\bR\setminus\bZ$, construct a certain fractal structure for $\Sing(b)$, and prove \eqref{Eq_dim_sing} of Theorem \ref{Thm_Sec3}. For this, we modify the strategy of \cite[Proposition 6.13]{CC16} appropriately.

Fix any $\del>0$ and consider two sequences $(\eps_i)_{i\geq 0}$ and $(N_i)_{i\geq 0}$ given by
\eqlabel{Eq_def_epsN}{
\eps_i=\eta 4^{-i}\qquad\text{and}\qquad N_i=\eta^{-1}5^{\frac{n}{\del}i}
}
for some constant $\eta>0$ to be determined.
Fix arbitrary $x_0 \in Q$ and choose $\eta>0$ small enough so that the set
\[
\wt{F}_{N_0}(x_0,\eps_0)=\{y=(p,q)\in F_{N_0}(x_0,\eps_0) : |bq|_\bZ > \eta2^{-1}\}
\]
is non-empty. It is possible due to $b\in\bR\setminus\bZ$. Pick $x_1\in \wt{F}_{N_0}(x_0,\eps_0)$ and define $Q_i$ with $i\geq 1$ recursively by
\[
Q_1=\{x_1\},\quad Q_{i+1} =\bigcup_{x\in Q_i} \wt{F}_{N_i}(x,\eps_i), 
\]
where
\[
\wt{F}_{N_i}(x,\eps_i) = \left\{y=(p,q)\in F_{N_i}(x,\eps_i): |bq|_\bZ > \eta 2^{-(i+1)}\right\}.
\]
Let $Q'=\bigcup_{i\geq 1}Q_i$ be a disjoint union. For each $x\in Q_i$, define $\sigma'(x)=\wt{F}_{N_i}(x,\eps_i)$ and $B(x)=B\left(\wh{x},\frac{\lam_1(x)}{2|x|}\right)$. Then we have 

\begin{prop}\label{Prop_fractal_sing}
The triple $(Q',\sigma',B)$ is a strictly nested self-similar structure covering a subset of $\Sing(b)$.
\end{prop}
\begin{proof}
It follows from Proposition \ref{Prop_CCfrac} that $(Q',\sigma',B)$ is a strictly nested self-similar structure.
For each $\sigma'$-admissible sequence $(x_i)$, since $\bigcap_i B(x_i)$ is a single point $\theta$, it is enough to show that $\theta \in \Sing(b)$. 

For any large enough $X\geq 1$, there exists $i\geq 1$ such that $\eta^{-1}2^i |x_i| < X \leq \eta^{-1}2^{i+1}|x_{i+1}|$.
As in the proof of \cite[Proposition 6.6]{CC16}, \cite[Lemma 4.2]{CC16} implies that each $x_i=(p_i,q_i)$ is a best approximate to $\theta$ (see \cite[Subsection 4.1]{CC16} for the definition).
Hence, it follows from \cite[Lemma 4.3(iii)]{CC16} that
\[
\|q_i \theta -p_i\| < 2\|q_i \wh{x}_{i+1}-p_i\|.
\]
Since $x_{i+1}\in \wt{F}_{N_{i}}(x_i,\eps_i)$, we have $|x_{i+1}|>(\frac{|x_i \wedge x_{i+1}|}{\eps_i})^{\frac{n}{n-1}}$, hence
\[
\|q_i \wh{x}_{i+1}-p_i\|_2 = \frac{|x_i\wedge x_{i+1}|}{|x_{i+1}|}<\eps_i  |x_{i+1}|^{-\frac{1}{n}}.
\]
It follows from $\|\cdot\| \leq \|\cdot\|_2$ that
\[
\|q_i \theta -p_i\| <2\eps_i  |x_{i+1}|^{-\frac{1}{n}} \leq 2^{1+\frac{1}{n}} \eps_i (\eta^{-1}2^i)^{\frac{1}{n}}X^{-\frac{1}{n}} 
\]
Since $x_i \in \wt{F}_{N_{i-1}}(x_{i-1},\eps_{i-1})$, we have $|bq_i|_\bZ >\eta 2^{-i}$, hence
\[
\|q_i \theta -p_i\| < 2^{1+\frac{1}{n}} \eps_i (\eta^{-1}2^{i})^{1+\frac{1}{n}}|bq_i|_\bZ X^{-\frac{1}{n}}\quad
\text{and}\quad |q_i|<\eta 2^{-i}X<|bq_i|_\bZ X.
\] 
Since $\eps_i=\eta 4^{-i}$ and $n\geq 2$, $\eps_i 2^{(1+\frac{1}{n})i}\to 0$ as $i\to \infty$. Therefore, $\theta \in \Sing(b)$.
\end{proof}

As in the proof of \cite[Proposition 6.13]{CC16}, using \cite[Lemmas 6.8 and 6.9]{CC16}, we can find constants $c_\rho>0$ and $c>0$ such that the function $\rho:Q'\to (0,1)$ given by 
\eqlabel{Eq_def_rho}{
\rho(x)=c_\rho \left(\frac{1}{N_i}\right)^{\frac{n+1}{n-1}}\eps_i^{\frac{2n}{n-1}+n} \quad \text{for each } x\in Q_i
}
has the property that for any $y \in \wt{F}_{N_i}(x,\eps_i)$ there are at most $c$ points $z\in \wt{F}_{N_i}(x,\eps_i)\setminus\{y\}$ such that $d(B(y),B(z))\leq \rho(x)\diam B(x)$.
Moreover, since $\rho(x)$ decreases with $i$, we have
\[
\rho(y)\diam B(y) < \rho(x) \diam B(x)
\] for all $y\in \sigma'(x)$ when $\eps_i$ is small enough. Therefore, the conditions \eqref{item_wlower_1} and \eqref{item_wlower_3} of Theorem \ref{thm:CCLow3.6} hold.

In order to show the condition \eqref{item_wlower_2} of Theorem \ref{thm:CCLow3.6}, we need the following proposition, which is a modification of \cite[Propsition 6.10]{CC16}.
\begin{prop}\label{Prop_lam_est}
    If $\eps_i$ is small enough, then for each $x\in Q_i$ and all real numbers $s$ and $t$ in $[0,2n]$, we have
    \[
    \sum_{y\in \wt{F}_{N_i}(x,\eps_i)} \frac{\wh{\lam}_1(y)^s}{|y|^t}\gg S_1(N_i,t)\frac{\eps_i^{s+(t-1)\frac{n}{n-1}}}{\wh{\lam}_n(x)^{(t-n)\frac{n}{n-1}}|x|^t},
    \]
    where
    \[
    S_1(N_i,t)=\sum_{k=1}^{N_i}\frac{1}{k^{1+(t-n)\frac{n}{n-1}}}\frac{\card(\cC_k'(x)\cap \Lam_x(\eps_i))}{\card(\cC_k'(x)\cap\Lam_x)}.
    \]
\end{prop}
\begin{proof}
    Observe that 
    \[\begin{split}
    \wt{F}_{N_i}(x,\eps_i)&=\left\{y=(p,q)\in F_{N_i}(x,\eps_i):|bq|_\bZ >\eta 2^{-(i+1)}\right\}\\
    &=\bigcup_{\al\in\Lam_x(\eps_i)\cap \cC_{N_i}(x)} \zeta(x,\al,\eps_i,b),
    \end{split}\]
    where 
    $\zeta(x,\al,\eps_i,b)=\{y=(p,q)\in \zeta(x,\al,\eps_i): |bq|_\bZ >\eta 2^{-(i+1)}\}.$ Therefore, it is enough to show 
    \eqlabel{eq_zeta_count}{
    \card \zeta(x,\al,\eps_i,b) \asymp \card \zeta(x,\al,\eps_i).
    }
    Then, the same proof of \cite[Propsition 6.10]{CC16} works in this case.
    
    In order to prove \eqref{eq_zeta_count}, write $x=(r,s)$ and fix a point $y=(p,q) \in \pi_x^{-1}(\al)\cap Q$. Since $\gcd(r,s)=1$, we have 
    $$\pi_x^{-1}(\al)\cap Q = \{(p + \ell r, q + \ell s) \cap Q: \ell \in \bZ\}.$$ 
    Since $|x\wedge y| = \|\al\|_2 |x|$, 
    \[
    \card \zeta(x,\al,\eps_i) = \card \left\{\ell \in \bZ : 
    (p+ \ell r, q + \ell s) \in Q, 
    \ga < q +\ell s < 2 \ga \right\},
    \] where $\ga=\ga_{x,\al,\eps_i} =\left(\frac{\|\al\|_2 |x|}{\eps_i}\right)^{\frac{n}{n-1}}$.
    It follows from the primitivity of $\al$ that $\gcd(p+\ell r, q+\ell s)=1$ for any $\ell \in\bZ$. 
    Hence, we have
    \[\begin{split}
     \card\zeta(x,\al,\eps_i) &= \card \left\{\ell \in \bZ : 
    \ga < q +\ell s < 2 \ga \right\};\\
    \card\zeta(x,\al,\eps_i,b) &=  \card\left\{\ell \in \bZ : 
    \ga < q +\ell s < 2 \ga, |b(q+\ell s)|_\bZ>\eta 2^{-(i+1)} \right\}.
    \end{split}\] 
    If $\ell\in \bZ$ is such that $\ga<q+\ell s<2\ga$ but $|b(q+\ell s)|_\bZ\leq \eta 2^{-(i+1)}$, then it follows from $x\in Q_i$, hence $|bs|_\bZ>\eta 2^{-i}$, that
    \[
    |b(q+(\ell\pm 1)s)|_\bZ \geq |bs|_\bZ - |b(q+\ell s)|_\bZ >\eta 2^{-i}-\eta2^{-(i+1)}=\eta 2^{-(i+1)}.
    \]
    Therefore, we have
    \[
     \frac{1}{2}\card \zeta(x,\al,\eps_i) \leq \card \zeta(x,\al,\eps_i,b) \leq \card \zeta(x,\al,\eps_i).
    \]
    This proves \eqref{eq_zeta_count}.
\end{proof}
We are now ready for the proof of \eqref{Eq_dim_sing} of Theorem \ref{Thm_Sec3}.
\begin{proof}[Proof of \eqref{Eq_dim_sing} of Theorem \ref{Thm_Sec3}]
We will prove 
\[
\dim_H \Sing(b) \geq \frac{n^2}{n-1}-\del.
\] Since $\del$ is arbitrary, this implies \eqref{Eq_dim_sing} of Theorem \ref{Thm_Sec3}. 
From the above observations and Proposition \ref{Prop_fractal_sing}, it is enough to show the condition \eqref{item_wlower_2} with $s=\frac{n^2}{n+1}-\del$ of Theorem \ref{thm:CCLow3.6} about $(Q',\sigma',B)$. 

First, it follows from Proposition~\ref{Prop_CC16_6.11} with $t=\frac{n+1}{n}s$, $\eps=\eps_i$, and $N=N_i$ that if $\eps_i$ is small enough, then
\eqlabel{Eq_S1_sum}{
S_1(N_i,t)\gg \sum_{k=1}^{N_i}\frac{1}{k^{1+(t-n)\frac{n}{n-1}}}.
} 

Now, we fix any $x\in Q_i$ and it follows from the proof of \cite[Corollary 6.12]{CC16} that $\wh{\lam}_1(x)\asymp\eps_{i-1}$ and $\wh{\lam}_n(x)^{-1}\gg \eps_{i-1}^{n-1}$.
Combining Proposition \ref{Prop_lam_est} and \eqref{Eq_S1_sum} with $t=\frac{n+1}{n}s=n-\frac{n+1}{n}\del$, together with the definitions of $\eps_i$, $N_i$, and $\rho$ from \eqref{Eq_def_epsN} and \eqref{Eq_def_rho}, we obtain
\[\begin{split}
\sum_{y\in \wt{F}_{N_i}(x,\eps_i)}& \frac{(\rho(y)\diam B(y))^s}{(\rho(x)\diam B(x))^s}\\
&=\left(\frac{N_i}{N_{i+1}}\right)^{\frac{n+1}{n-1}s}\left(\frac{\eps_{i+1}}{\eps_i}\right)^{(\frac{2n}{n-1}+n)s} \frac{|x|^{t}}{\wh{\lam}_1(x)^s} \sum_{y\in \wt{F}_{N_i}(x,\eps_i)} \frac{\wh{\lam}_1(y)^s}{|y|^t}\\
&\gg \eps_i^{s+(t-1)\frac{n}{n-1}}\eps_{i-1}^{(t-n)n-s} \sum_{k=1}^{N_i}\frac{1}{k^{1+(t-n)\frac{n}{n-1}}}\\
&\gg  4^{-i\left(n-\frac{n+1}{n-1}\del -(n+1)\del\right)}5^{i\frac{n(n+1)}{n-1}}\\
&\gg (5/4)^{in}.
\end{split}\]
Therefore, if $i$ is large enough, then the sum exceeds $c+1$. This proves the condition \eqref{item_wlower_2} with $s=\frac{n^2}{n+1}-\del$ of Theorem \ref{thm:CCLow3.6}. 
\end{proof}

\subsection{Fractal structure for $\DI_\eps (b)$}
Now we assume that $b\in\bQ\setminus\bZ$. Similar to the previous subsection, we will construct a certain fractal structure for $\DI_\eps (b)$ and prove \eqref{Eq_dim_DI} of Theorem \ref{Thm_Sec3}. 

Let $b=c/d$ with $\gcd(c,d)=1$ and $d\geq 2$. The key property we will use is the following: for any $k\in\bZ$,
\eqlabel{Eq_b_prop}{
|bk|_\bZ < 1/d \quad \iff \quad |bk|_\bZ =0. 
}
Given $\eps>0$, $x\in Q$, and $N\in\bN$, define
\[
F_{N,b} (x,\eps) = \{y=(p,q)\in F_N(x,\eps): |bq|_\bZ \geq 1/d\},
\]
and define 
\[
\sigma_{\eps,N,b}(x)=F_{N,b}(x,\eps),\quad 
Q_{\eps,N,b}=\bigcup_{x\in Q}\sigma_{\eps,N,b}(x),\quad 
B(x)=B\left(\wh{x},\frac{\lam_1(x)}{2|x|}\right).
\]
Similar to Proposition \ref{Prop_fractal_sing}, we have
\begin{prop}\label{Prop_fractal}
 The triple $(Q_{\eps,N,b},\sigma_{\eps,N,b},B)$ is a strictly nested self-similar structure covering a subset of $\DI_{2d^{1+1/n}\eps}(b)$ provided $\eps$ is small enough.
\end{prop}
\begin{proof}
It follows from Proposition \ref{Prop_CCfrac} that $(Q_{\eps,N,b},\sigma_{\eps,N,b},B)$ is a strictly nested self-similar structure.
For each $\sigma_{\eps,N,b}$-admissible sequence $(x_i)$, since $\bigcap_i B(x_i)$ is a single point $\theta$, it is enough to show that $\theta \in \DI_{2\eps}(b)$. 

For any large enough $X\geq 1$, there exists $i\geq 1$ such that $d|x_i| <  X \leq d|x_{i+1}|$.
Following the proof of Proposition \ref{Prop_fractal_sing}, if we write $x_i=(p_i,q_i)$, it follows from $x_{i+1}\in F_{N,b}(x_i,\eps)$ that
\[
\|q_i \theta -p_i\| <2\eps  |x_{i+1}|^{-\frac{1}{n}} \leq 2\eps d^{\frac{1}{n}} X^{-\frac{1}{n}}. 
\]
Since $x_i \in F_{N,b}(x_{i-1},\eps)$, we have $|bq_i|_\bZ \geq 1/d$, hence
\[
\|q_i \theta -p_i\| < 2 \eps d^{1+\frac{1}{n}}|bq_i|_\bZ X^{-\frac{1}{n}}\quad
\text{and}\quad |q_i|<\frac{1}{d}X\leq |bq_i|_\bZ X.
\] 
Therefore, $\theta \in \DI_{2d^{1+1/n}\eps}(b)$.
\end{proof}

As before, in order to show the condition \eqref{item_wlower_2} of Theorem \ref{thm:CCLow3.6}, we need the following proposition.

\begin{prop}\label{Prop_count}
If $\eps$ is small enough, then for any $N\in\bN$, each $x\in Q_{\eps,N,b}$, and all real numbers $s$ and $t$ in $[0,2n]$, we have
\[
\sum_{y\in F_{N,b}(x,\eps)}\frac{\wh{\lam}_1(y)^s}{|y|^t}\gg S_1(N,t)\frac{\eps^{s+(t-1)\frac{n}{n-1}}}{\wh{\lam}_n(x)^{(t-n)\frac{n}{n-1}}|x|^t},
\]
where
\[
S_1(N,t) = \sum_{k=1}^N \frac{1}{k^{1+(t-n)\frac{n}{n-1}}}\frac{\card(\cC_k'(x)\cap\Lam_x(\eps))}{\card(\cC_k'(x)\cap\Lam_x)}.
\]
\end{prop}
\begin{proof}
Similar to the proof of Proposition \ref{Prop_lam_est}, since 
    \[\begin{split}
    F_{N,b}(x,\eps)&=\left\{y=(p,q)\in F_{N}(x,\eps):|bq|_\bZ \geq 1/d\right\}\\
    &=\bigcup_{\al\in\Lam_x(\eps)\cap \cC_{N}(x)} \zeta(x,\al,\eps,b),
    \end{split}\]
    where 
    $\zeta(x,\al,\eps,b)=\{y=(p,q)\in \zeta(x,\al,\eps): |bq|_\bZ \geq 1/d\},$ it is enough to show that
    \eqlabel{eq_zeta_count_2}{
    \card \zeta(x,\al,\eps,b) \asymp \card \zeta(x,\al,\eps)
    } for all small enough $\eps>0$.
    Then, the same proof of \cite[Propsition 6.10]{CC16} works in this case.

Following the proof of Proposition \ref{Prop_lam_est}, we have
    \[\begin{split}
     \card\zeta(x,\al,\eps) &= \card \left\{\ell \in \bZ : 
    \ga < q +\ell s < 2 \ga \right\};\\
    \card\zeta(x,\al,\eps,b) &=  \card\left\{\ell \in \bZ : 
    \ga < q +\ell s < 2 \ga, |b(q+\ell s)|_\bZ\geq 1/d \right\},
    \end{split}\] 
where $x=(r,s)$, $y=(p,q)\in\pi_x^{-1}(\al)\cap Q$, and $\ga=\ga_{x,\al,\eps} =\left(\frac{\|\al\|_2 |x|}{\eps}\right)^{\frac{n}{n-1}}$.
 If $\ell\in \bZ$ is such that $\ga<q+\ell s<2\ga$ but $|b(q+\ell s)|_\bZ< 1/d$, then $|b(q+\ell s)|_\bZ=0$ by \eqref{Eq_b_prop}. Thus, it follows from $x\in Q_{\eps,N,b}$, hence $|bs|_\bZ\geq 1/d$, that
    \[
    |b(q+(\ell\pm 1)s)|_\bZ \geq |bs|_\bZ - |b(q+\ell s)|_\bZ \geq 1/d.
    \]
    Therefore, we have
    \[
     \frac{1}{2}\card \zeta(x,\al,\eps) \leq \card \zeta(x,\al,\eps,b) \leq \card \zeta(x,\al,\eps).
    \]
    This proves \eqref{eq_zeta_count_2}.
\end{proof}

\begin{proof}[Proof of \eqref{Eq_dim_DI} of Theorem \ref{Thm_Sec3}]
Using Proposition \ref{Prop_fractal} and Proposition \ref{Prop_count}, \eqref{Eq_dim_DI} of Theorem \ref{Thm_Sec3} follows by the same proof of \cite[Corollary 6.12]{CC16}.
\end{proof}

\section{Dimension estimates for $\Bad_A(\infty)$}\label{sec4}
In this section, we will prove Theorem \ref{Thm_3}. It is well-known that $A\in M_{m,n}(\bR)$ is singular if and only if $^tA\in M_{n,m}(\bR)$ is singular. This can be proved easily from the following transference theorem.
\begin{thm}\cite[Theorem \rom{2} in Chapter \rom{5}]{Cas57}\label{thm_trasf}
Let $0<C<1\leq X$ be constants.
Suppose that there is $\mb{q}\in\bZ^n\setminus\{0\}$ such that
\[
\|A\mb{q}\|_\bZ \leq C \quad\text{and}\quad \|\mb{q}\| \leq X.
\]
Then there is $\mb{y}\in \bZ^m\setminus\{0\}$ such that
\[
\|^tA\mb{y}\|_\bZ \leq D \quad\text{and}\quad \|\mb{y}\|\leq U,
\] where 
\[
D=(m+n-1)X^{\frac{1-m}{m+n-1}}C^{\frac{m}{m+n-1}},\quad U =(m+n-1)X^{\frac{n}{m+n-1}}C^{\frac{1-n}{m+n-1}}.
\]
\end{thm}
We also obtain the following similar result for very singular matrices.
\begin{prop}\label{Prop_transf} For $A \in M_{m,n}(\bR)$,
$A$ is very singular if and only if ${^t}A$ is very singular.
\end{prop}
\begin{proof}
It is enough to show one direction by the symmetry of the argument. Assume $A$ is very singular, that is, $\wh{w}(A)>\frac{n}{m}$. 
Choose $\del>0$ such that $\wh{w}(A)>\frac{n}{m}+\del$, that is, for all sufficiently large $X$ there is $\mb{q}\in\bZ^n$ such that
\[
\|A\mb{q}\|_\bZ <X^{-\frac{n}{m}-\del}\quad\text{and}\quad 0<\|\mb{q}\|<X.
\]
Using Theorem \ref{thm_trasf}, there is $\mb{y}\in\bZ^m$ such that
\[\begin{split}
\|^tA\mb{y}\|_\bZ &\leq (m+n-1)X^{\frac{1-m}{m+n-1}-\frac{n}{m+n-1}-\del\frac{m}{m+n-1}}=(m+n-1)X^{-1-\del\frac{m}{m+n-1}};\\ 
\|\mb{y}\|&\leq (m+n-1)X^{\frac{n}{m+n-1}-\frac{n}{m}\frac{1-n}{m+n-1}-\del\frac{1-n}{m+n-1}}=(m+n-1)X^{\frac{n}{m}+\del\frac{n-1}{m+n-1}}. 
\end{split}
\]
By taking $Y=(m+n)X^{\frac{n}{m}+\del\frac{n-1}{m+n-1}}$, we have
\[
\|^tA \mb{y}\|_\bZ < Y^{-\frac{m}{n}-\del'}\quad \text{and}\quad 0<\|\mb{y}\|<Y
\] for some small $\del'>0$. This proves that $^t A$ is very singular.
\end{proof}

Now we recall the result in \cite{Kim}. Assume that the subgroup ${^{t}A}\bZ^m + \bZ^n$ of $\bR^n$ has maximal rank $m+n$ over $\bZ$. Following \cite[Section 3]{BL05}, there exists a sequence of best approximations $(\mb{y}_k)_{k\geq 1}$ in $\bZ^m$ for $^t A$. Denote $Y_k = \|\mb{y}_k\|$, $M_k = \|{^t A}\mb{y}_k\|_\bZ$, and
\[
\gamma_k = \max\left(\left(Y_{k}^{\frac{m}{n}}M_{k-1}\right)^{\frac{n}{m+n}}, \left(Y_{k+1}^{\frac{m}{n}}M_{k}\right)^{\frac{n}{m+n}}\right).
\]
For any $\al>0$, let
\[
B_\al(\{\gamma_k\}) = \{\mb{b}\in[0,1]^m : |\mb{b}\cdot\mb{y}_k|_\bZ > \al \ga_k \text{ for all large enough }k\geq 2\}.
\]
\begin{prop}\cite[Proposition 5.1]{Kim}\label{Prop_subset} For any $\al>0$,
\[B_\al(\{\gamma_k\})\subset \Bad_A\left(\frac{\al-n}{m}\right).\]
\end{prop}

Using this proposition, we are able to prove Theorem \ref{Thm_3}.
\begin{proof}[Proof of Theorem \ref{Thm_3}]
Since the theorem is trivial if $\wh{w}( ^tA) \leq \frac{m}{n}$, we may assume that $\wh{w}( ^tA) > \frac{m}{n}$, i.e. $^tA$ is very singular.

We first consider the case that the rank of the subgroup ${^{t}A}\bZ^m + \bZ^n$ over $\bZ$ is strictly less than $m+n$. If so, we can choose a nonzero $\mb{y}\in\bZ^m$ such that $^tA\mb{y}\in\bZ^n$. Fix such nonzero $\mb{y}\in\bZ^m$. Note that $\wh{w}({^t}A)=\infty$ in this case. Observe that if $|\mb{b}\cdot\mb{y}|_\bZ >0$, then $^tA$ is singular for $\mb{b}$. By Theorem \ref{Thm_1}, the set $\{\mb{b}\in\bR^m : |\mb{b}\cdot\mb{y}|_\bZ >0\}$ is contained in $\Bad_A(\infty)$. Hence, the complement $\bR^m\setminus \Bad_A(\infty)$ is contained in the set $\{\mb{b}\in\bR^m : |\mb{b}\cdot\mb{y}|_\bZ =0\}$, which is a countable union of $(m-1)$-dimensional hyperplanes. This implies that $\dim_H (\bR^m\setminus \Bad_A(\infty))\leq m-1$.

Now assume that the subgroup ${^{t}A}\bZ^m + \bZ^n$ has maximal rank $m+n$ over $\bZ$. Since $\mb{b}\in \Bad_A(\infty)$ if and only if $\mb{b}+\mb{x}\in \Bad_A(\infty)$ for any $\mb{x}\in\bZ^m$, by proposition \ref{Prop_subset}, it is enough to show that for any $\al>0$,
 \eqlabel{Eq_dimgoal}{\dim_H ([0,1]^m\setminus B_\al(\{\ga_k\})) \leq m-\frac{\wh{w}(^tA)-\frac{m}{n}}{\wh{w}(^tA)+1}.}
 Observe that
\[
[0,1]^m\setminus B_\al(\{\ga_k\}) = \limsup_{k\to\infty}\ \{\mb{b}\in[0,1]^m : |\mb{b}\cdot\mb{y}_k|_\bZ \leq \al \ga_k\},
\]
and for each $k$, the set $\{\mb{b}\in[0,1]^m : |\mb{b}\cdot\mb{y}_k|_\bZ \leq \al \ga_k\}$ can be covered with $C_1\frac{Y_k^m}{\ga_k^{m-1}}$ balls of radius $C_2\frac{\ga_k}{Y_k}$ for some absolute constants $C_1,C_2>0$.
Using Hausdorff-Cantelli Lemma \cite[Lemma 3.10]{BD99}, it follows that for any $0\leq s\leq m$, the $s$-dimensional Hausdorff measure
\[
\cH^s ([0,1]^m\setminus B_\al(\{\ga_k\})) =0 \quad\text{if}\quad \sum_{k\geq 2}\frac{Y_k^m}{\ga_k^{m-1}}\left(\frac{\ga_k}{Y_k}\right)^{s} <\infty.
\]
Since $^tA$ is very singular, for any $0<\del<\wh{w}(^tA)-\frac{m}{n}$, we have $Y_{k+1}^{\frac{m}{n}+\del}M_k < 1$ for all sufficiently large $k$. Since $\ga_k \leq Y_k^{-\frac{\del n}{m+n}}$ for all sufficiently large $k$, it follows that
\[
\sum_{k\geq 2}\frac{Y_k^m}{\ga_k^{m-1}}\left(\frac{\ga_k}{Y_k}\right)^{s} \leq \sum_{k\geq 2}Y_k^{m-s-(s-m+1)\frac{\del n}{m+n}}
\] for any $m-1\leq s\leq m$.
For any $m-\frac{\del n}{m+n+\del n}<s<m$, since $m-s- (s-m+1)\frac{\del n}{m+n}<0$ and $Y_k$ increases at least geometrically (see \cite[Lemma 1]{BL05}), it follows that 
$$\sum_{k}Y_k^{m-s-(s-m+1)\frac{\del n}{m+n}}<\infty.$$ 
Hence we have $\dim_H ([0,1]^m\setminus B_\al(\{\ga_k\})) \leq s$. Taking $s\to m-\frac{\del n}{m+n+\del n}$ and $\del \to \wh{w}(^tA)-\frac{m}{n}$, we finally have \eqref{Eq_dimgoal}.
\end{proof}

\section{Homogeneous dynamics}\label{sec5}
In this section, we discuss a relation between the infinitely badly approximable property and the dynamical property mentioned in Subsection \ref{subsec_1.3}. We first prove the following implication.
\begin{prop}\label{Prop_transf} 
For any $x\in \cX$, 
\[\Del(a_t x)\to\infty \text{ as } t\to\infty \implies \Del_0(a_t\pi(x))\to 0 \text{ as } t\to\infty.\]
\end{prop}
\begin{proof}
Given a lattice $\Lam$ in $\bR^{m+n}$, let us denote by $\lam_j(\Lam)$ be the $j$-th successive minimum of $\Lam$, i.e. the infimum of $\lam$ such that the ball in $\bR^{m+n}$ of radius $\lam$ around $0$ contains $j$ independent vectors of $\Lam$. 

Given $x\in \cX$ and $t\in\bR$, let $\mb{v}_1,\dots,\mb{v}_{m+n}$ be independent vectors of the lattice $a_t\pi(x)$ satisfying $\|\mb{v}_i\|\leq \lam_{m+n}(a_t\pi(x))$ for all $i=1,\dots,m+n$. Since the parallelepiped $\Pi=\{\sum_{i}\al_i\mb{v}_i : \forall i, -1\leq \al_i \leq 1 \}$ contains a fundamental domain of the lattice $a_t\pi(x)$, there is a point $\mb{w}$ of the grid $a_t x$ such that $\mb{w}+\Pi$ contains the origin $0$. Hence, it follows that
\[
\Del(a_t x) \leq \|\mb{w}\| \leq \sum_{i=1}^{m+n} \|\mb{v}_i\| \leq (m+n)\lam_{m+n}(a_t\pi(x)).
\] 
Since $\lam_1(a_t\pi(x)) \cdots \lam_{m+n}(a_t\pi(x)) \asymp 1$ by Minkowski's second theorem,
we have $\Del(a_t x) \ll \lam_1(a_t\pi(x))^{-(m+n-1)}$. Therefore, if $\Del(a_t x)\to\infty$ as $t\to \infty$, then $\Del_0(a_t\pi(x))\to 0$ as $t\to\infty$.
\end{proof}

In order to prove Theorem \ref{Thm_4}, we need certain invariance properties for the divergence of the function $\Del$.
Let us denote elements of $\ASL_{m+n}(\bR)$ by $\idist{g,\mb{v}}$, where $g\in\SL_{m+n}(\bR)$ and $\mb{v}\in\bR^{m+n}$.
Consider the action $a_t$ on $\ASL_{m+n}(\bR)$ by the left multiplication of $\idist{a_t,0}$ on $\ASL_{m+n}(\bR)$. 
The \textit{expanding horospherical subgroup} of $\ASL_{m+n}(\bR)$ with respect to $(a_t)_{t\geq 0}$ is given by 
\[
U=\left\{\left\langle\begin{pmatrix} I_m & A  \\ 0 & I_n  \end{pmatrix},\begin{pmatrix} \mb{b} \\ 0\end{pmatrix}\right\rangle : A\in M_{m,n}, \mb{b}\in\bR^m\right\}.
\] 
On the other hand, the \textit{nonexpanding horoshperical subgroup} of $\ASL_{m+n}(\bR)$ with respect to $(a_t)_{t\geq 0}$ is given by
\[
P =\left\{\left\langle\begin{pmatrix} S & 0  \\ R & Q  \end{pmatrix},\begin{pmatrix} 0 \\ \mb{c}\end{pmatrix}\right\rangle : 
\begin{aligned}
    &S\in M_{m,m},\ Q\in M_{n,n},\ R\in M_{n,m},\\
    &\mb{c}\in\bR^n,\ \det S \det Q =1
\end{aligned}\right\}.
\]
\begin{prop}\label{Prop_inv} For any $x\in \cX$ and any $p\in P$,
\[
\Del(a_tx)\to\infty \text{ as } t\to\infty \iff \Del(a_t px)\to\infty \text{ as } t\to\infty.
\]
\end{prop}
\begin{proof}
The following is a key observation: for $p=\left\langle\begin{psmallmatrix} S & 0  \\ R & Q  \end{psmallmatrix},\begin{psmallmatrix} 0 \\ \mb{c}\end{psmallmatrix}\right\rangle \in P$,
\[
a_t \left\langle\begin{psmallmatrix} S & 0  \\ R & Q  \end{psmallmatrix},\begin{psmallmatrix} 0 \\ \mb{c}\end{psmallmatrix}\right\rangle a_{-t} 
= \left\langle\begin{psmallmatrix} S & 0  \\ e^{-\left(\frac{1}{m}+\frac{1}{n}\right)t}R & Q  \end{psmallmatrix},\begin{psmallmatrix} 0 \\ e^{-\frac{1}{n}t}\mb{c}\end{psmallmatrix}\right\rangle.
\] 
Therefore, for any $t\geq 0$ and any point $\mb{v}$ of the grid $x$, 
\[
\|a_t p\mb{v}\| = \|a_t p a_{-t} a_t\mb{v}\| \asymp \|a_t \mb{v}\|, 
\]
where the implied constant depends only on $p$. This concludes the proof of the proposition.
\end{proof}

\begin{proof}[Proof of Theorem \ref{Thm_4}]
Assume $m=1$ and $n\geq 2$. By the upper bound \eqref{Eq_upper}, it is enough to show that 
\[
\dim_H\{x\in \cX : \Del(a_tx)\to\infty\ \text{as}\ t\to\infty\} \geq \dim_H \cX -\frac{n}{n+1}.
\]
Since the product map $P\times U \to \ASL_{m+n}(\bR)$ is a local diffeomorphism, every element of a neighborhood of identity of $\ASL_{m+n}(\bR)$ can be written as $pu$ where $p\in P$ and $u\in U$. Therefore, using \eqref{Eq_char}, Proposition \ref{Prop_inv}, and Theorem \ref{Thm_2}, we have
\[\begin{split}
\dim_H\{x\in \cX : \Del(a_tx)\to\infty\ \text{as}\ t\to\infty\} &\geq \dim_{H} \Bad(\infty) + \dim_H P\\
& = \dim_H \cX - \frac{n}{n+1}.
\end{split}\]
\end{proof}

\bibliography{ref}
\bibliographystyle{amsalpha}

\end{document}